\documentclass[11pt]{article}

\usepackage{fullpage}
\usepackage{amsmath,amsthm,amsfonts,dsfont}
\usepackage{amssymb,latexsym,graphicx}
\usepackage{palatino}
\usepackage{mathpazo}
\usepackage{stmaryrd}
\usepackage{mathtools}
\usepackage{hyperref}
\usepackage{graphicx}
\usepackage{caption}
\usepackage{subcaption}

\newtheorem{theorem}{Theorem}[section]

\newtheorem{proposition}[theorem]{Proposition}

\newtheorem{corollary}[theorem]{Corollary}

\newcommand{\R}{\ensuremath{\mathbb{R}}}
\newcommand{\Z}{\ensuremath{\mathbb{Z}}}

\renewcommand{\epsilon}{\varepsilon}

\renewcommand{\vec}[1]{\ensuremath{\mathbf{#1}}}

\newcommand{\basis}{\ensuremath{\mathbf{B}}}

\DeclareMathOperator*{\expect}{\mathbb{E}}

\newcommand{\grad}{\nabla}

\newcommand{\scarequotes}[1]{``#1''}

\newcommand{\M}{\mathcal{M}}
\newcommand{\NN}{\mathcal{N}}

\makeatletter
\def\imod#1{\allowbreak\mkern8mu({\operator@font mod}\,\,#1)}
\makeatother

\newcommand{\lat}{\mathcal{L}}
\DeclarePairedDelimiter\inner{\langle}{\rangle}
\DeclarePairedDelimiter\length{\lVert}{\rVert}

\begin{document}

\title{
An Inequality for Gaussians on Lattices}
\author{
Oded Regev\thanks{Courant Institute of Mathematical Sciences, New York
 University.}
~\thanks{Supported by the Simons Collaboration on Algorithms and Geometry and by the National Science Foundation (NSF) under Grant No.~CCF-1320188. Any opinions, findings, and conclusions or recommendations expressed in this material are those of the authors and do not necessarily reflect the views of the NSF.}\\
\and
Noah Stephens-Davidowitz\footnotemark[1]~\thanks{Supported by the National Science Foundation (NSF) under Grant No.~CCF-1320188. Any opinions, findings, and conclusions or recommendations expressed in this material are those of the authors and do not necessarily reflect the views of the NSF.}\\
\texttt{noahsd@cs.nyu.edu}
}
\date{}
\maketitle

\begin{abstract}
We show that for any lattice $\lat \subseteq \R^n$ and vectors $\vec{x}, \vec{y} \in \R^n$, 
\[
\rho(\lat + \vec{x})^2 \rho(\lat + \vec{y})^2 \leq \rho(\lat)^2 \rho(\lat + \vec{x} + \vec{y}) \rho(\lat + \vec{x} - \vec{y})
\; ,
\]
where $\rho$ is the Gaussian
mass function $\rho(A) := \sum_{\vec{w} \in A} \exp(-\pi \length{\vec{w}}^2)$. We show a number of applications, including bounds on the moments of the discrete Gaussian distribution, various monotonicity properties of the heat kernel on flat tori, and a 
positive correlation inequality for Gaussian measures on lattices.
\end{abstract}

\section{Introduction}
A lattice $\lat\subset \R^n$ is the set of all integer linear combinations of $n$ linearly independent vectors $\basis = (\vec{b}_1, \ldots, \vec{b}_n )$. 
For any $s>0$, we define the function $\rho_s : \R^n \rightarrow\R$ as
\[
\rho_s(\vec{x}) = \exp(-\pi \length{\vec{x}}^2/s^2) \; .
\] 
For a discrete set $A \subset \R^n$ we define $\rho_s(A)=\sum_{\vec{w}\in A} \rho_s(\vec{w})$. 
The \emph{discrete Gaussian distribution} over a lattice coset $\lat + \vec{x}$ with parameter $s$, $D_{\lat + \vec{x},s}$, is the probability distribution over $\lat + \vec{x}$ that assigns probability  
\[
\frac{\rho_s(\vec{w})}{\rho_s(\lat + \vec{x})}
\]
to each vector $\vec{w} \in \lat + \vec{x}$.
(See Figure~\ref{fig:DGS}.)
The \emph{periodic Gaussian function} over $\lat$ with parameter $s$ is
\[
f_{\lat, s}(\vec{x}) := \frac{\rho_s(\lat + \vec{x})}{\rho_s(\lat)} \; . \]
(See Figure~\ref{fig:periodic}.)
When $s = 1$, we write $\rho(\vec{x})$, $D_{\lat + \vec{x}}$, and $f_\lat(\vec{x})$. 

\begin{figure}[ht]
\begin{center}
\begin{subfigure}[b]{0.4\textwidth}
                \includegraphics[width=\textwidth]{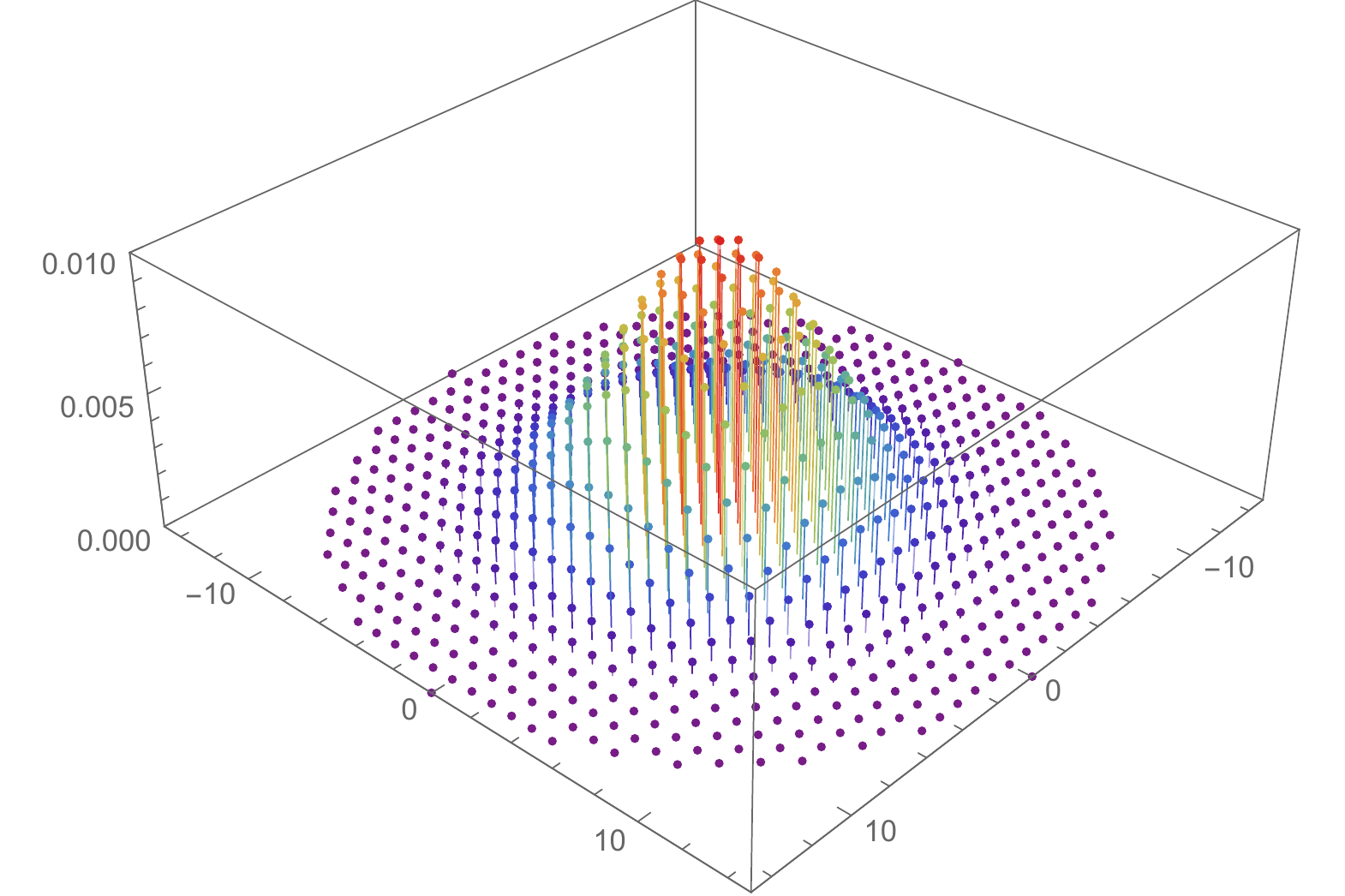}
\caption{\label{fig:DGS} 
The discrete Gaussian distribution on $\Z^2$ with parameter $s=10$.}
        \end{subfigure} \qquad
\begin{subfigure}[b]{0.4\textwidth}
\includegraphics[width= 0.8\textwidth]{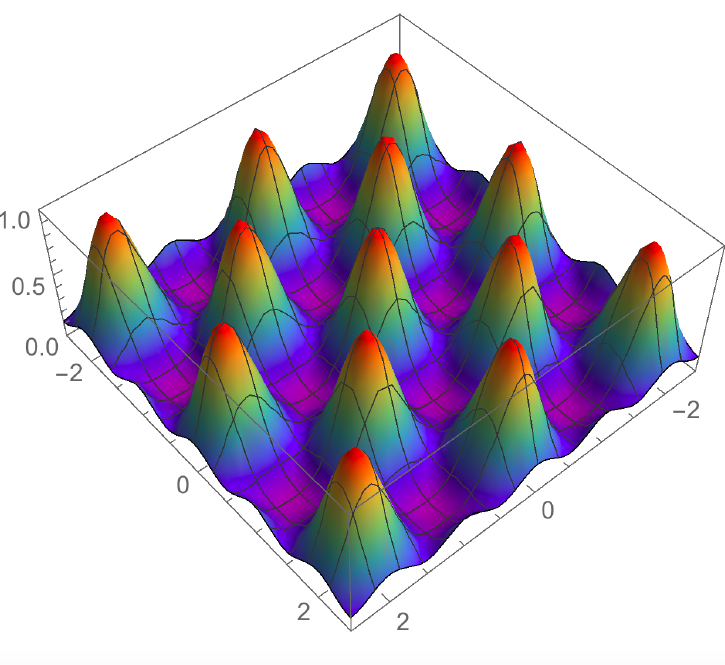}
\caption{\label{fig:periodic} 
The periodic Gaussian function on the lattice spanned by $(1,1) $ and $(1,-1)$ with parameter $s=3/4$.}
        \end{subfigure}
\caption{ }
\end{center}
\end{figure}

These objects appear in several guises in mathematics and are well studied.
For example, $\rho(\lat + \vec{x})$ is the Riemann theta function in a dual form (see, e.g.,~\cite{Mumford}) and was studied in connection with the Riemann zeta function~\cite{Riemann, BianePY01}; it can also be seen as the heat kernel on the flat torus $\R^n/\lat$; 
it played an instrumental role in proving tight transference theorems for lattices~\cite{banaszczyk}; it was used to construct bilipschitz embeddings
of flat tori into a Hilbert space~\cite{HavivR13}; and the authors recently used it to bound the number of short lattice points~\cite{rev_mink}. Both $D_{\lat + \vec{x}}$ and $f_\lat$ have also played an important role in recent years
in computer science, especially in cryptographic applications of lattices (e.g.,~\cite{MR04,GPV08}).
Our motivation comes from attempts to improve upon the current fastest known algorithms for the main computational problems
on lattices, the Shortest Vector Problem~\cite{AggarwalDRS15} and the Closest Vector Problem~\cite{ADS15}. (Both algorithms rely on special cases of the main result of this work.)

In spite of their importance, there is still a lot that we do not know about $\rho_s(\lat + \vec{x})$, $f_{\lat, s}(\vec{x})$, and $D_{\lat + \vec{x}, s}$. 
In this work, we prove several basic inequalities concerning these objects, 
as described below. 
All of these inequalities
follow without too much effort
from one main inequality (Theorem~\ref{thm:awesometheorem}), which is closely related to Riemann's theta relations (see~\cite{Mumford}). 
Namely, in terms of the periodic Gaussian $f_{\lat}(\vec{x})$, our main inequality says that 
\[
f_{\lat}(\vec{x})^2 f_{\lat}(\vec{y})^2 \leq f_{\lat}(\vec{x} + \vec{y}) f_{\lat}(\vec{x} - \vec{y})
\; .
\]
Note that the Gaussian function $\rho(\vec{x})$ over $\R^n$ satisfies the ``rotation'' identity
\[
\rho(\vec{x})^2 \rho(\vec{y})^2 = \rho(\vec{x} + \vec{y}) \rho(\vec{x} - \vec{y})
\; ,
\]
so that our main inequality can be viewed as a relaxation of this identity to the periodic case.
From this (perhaps rather opaque) inequality, we derive many natural statements concerning $\rho_s(\lat + \vec{x})$, $f_{\lat, s}(\vec{x})$, and $D_{\lat + \vec{x}, s}$.

First, we show in Corollary~\ref{cor:secondmoment} that the covariance of $D_{\lat + \vec{x}}$ is minimized when $\vec{x} = \vec0$,
answering a natural question communicated to us by Dadush~\cite{priv:Daniel}. 
(We note in passing that 
closely related questions are still open, e.g., 
whether $\expect_{\vec{w} \sim D_{\lat + \vec{x}}}[\length{\vec{w}}]$ is minimized when $\vec{x} = \vec0$.)
Along the way, we derive an interesting inequality concerning the \scarequotes{shape} of $f_{\lat}(\vec{x})$ (Proposition~\ref{prop:Hess}). We also analyze the fourth moment, showing in particular that the discrete Gaussian is \scarequotes{leptokurtic} (Proposition~\ref{prop:fourthmoment})---i.e., its kurtosis is at least that of the \emph{continuous} Gaussian distribution.

Second, in Section~\ref{sec:monotone}, we show various monotonicity results concerning $f_{\lat, s}$, answering a natural open question due to Price~\cite{PriceMO1} in the affirmative. In particular, in Proposition~\ref{prop:monotones} we show that $f_{\lat, s}$ is monotonic in $s$ and in 
Proposition~\ref{prop:monocovariance}, we extend this to the non-spherical Gaussian case. 
(Recently, Price showed how to derive from this an analogous monotonicity result for Abelian Cayley graphs~\cite{Price16}.
A further extension to arbitrary Cayley graphs, previously 
suggested by Peres~\cite{PeresPrivate}, turns out to be false~\cite{RegevShinkar}.)
Additionally, in Proposition~\ref{prop:monotonesublattice}, we show that $f_{\lat, s}$ is monotonic under taking sublattices of $\lat$. 

Finally, in Section~\ref{sec:correlation}, we show that sublattices of a lattice $\lat$ are positively correlated under the normalized Gaussian measure on $\lat$. This result answers another open question due to Price~\cite{PriceMO2}, and
was recently used by him in his work on cohomology~\cite{PricePaper}. 
It has a (possibly superficial) resemblance to the recently proven Gaussian correlation conjecture on symmetric convex bodies~\cite{RoyenGCC}.  In fact, we note in passing that our main inequality can also be viewed as a correlation result. In particular, it shows that $\cos(2\pi \inner{\vec{w}, \vec{x}})$ and $\cos(2\pi \inner{\vec{w}, \vec{y}})$ are positively correlated when $\vec{w}$ is sampled from $D_\lat$. (See Eq.~\ref{eq:coscorrelation}.)

\section{The main inequality}

The following is our main theorem. The proof is essentially a combination 
of a certain identity related to Riemann's theta relations (see~\cite[Chapter 1, Section 5]{Mumford}) and
the Cauchy-Schwarz inequality.

\begin{theorem}
\label{thm:awesometheorem}
For any lattice $\lat \subset \R^n$ and any two vectors $\vec{x}, \vec{y} \in \R^n$, we have
\[
\rho(\lat + \vec{x})^2 \rho(\lat + \vec{y})^2 \leq \rho(\lat)^2 \rho(\lat + \vec{x} + \vec{y}) \rho(\lat + \vec{x} - \vec{y})
\; .
\]
\end{theorem}
\begin{proof}
Let $\lat^{\oplus 2} := \lat \oplus \lat$. We can then write $\rho(\lat + \vec{x}) \rho(\lat + \vec{y}) = \rho(\lat^{\oplus 2} + (\vec{x}, \vec{y}))$.
Consider the $2n \times 2n$ matrix
\[
T := \left( \begin{array}{rr}
I_n & I_n \\
I_n & -I_n
\end{array}\right)
\; ,
\]
where $I_n$ is the $n\times n$ identity matrix.
Note that $T/\sqrt{2}$ is an orthogonal matrix so that $\length{T\vec{v}} = \sqrt{2} \length{\vec{v}}$ for any $\vec{v} \in \R^{2n}$. We therefore have 
\begin{equation}
\label{eq:rotated}
\rho(\lat + \vec{x}) \rho(\lat + \vec{y}) = \rho_{\sqrt{2}}\big(T(\lat^{\oplus 2} + (\vec{x}, \vec{y}))\big) = \rho_{\sqrt{2}}\big(T \lat^{\oplus 2} + (\vec{x} + \vec{y}, \vec{x} - \vec{y})\big)
\; .
\end{equation}
For any $\vec{z} := (\vec{z}_1, \vec{z}_2) \in \lat^{\oplus 2}$, we have
$T \vec{z} = (\vec{w}_1, \vec{w}_2)$
where $\vec{w}_1 :=  \vec{z}_1 + \vec{z}_2$ and $\vec{w}_2 := \vec{w}_1 - 2 \vec{z}_2$.
It follows that
\begin{align*}
T \lat^{\oplus 2} &= \{ (\vec{w}_1, \vec{w}_2) \in \lat^{2} : \vec{w}_1 \equiv \vec{w}_2 \bmod 2 \lat\} \\
&= \bigcup_{\vec{c} \in \lat/(2\lat)} (2 \lat + \vec{c})^{2}
\; ,
\end{align*}
where the union is disjoint.
Plugging in to Eq.~\eqref{eq:rotated}, we have
\begin{equation}
\rho(\lat + \vec{x}) \rho(\lat + \vec{y}) = \sum_{\vec{c} \in \lat/(2\lat)} \rho_{\sqrt{2}}(2 \lat + \vec{c} + \vec{x} + \vec{y}) \cdot \rho_{\sqrt{2}}(2 \lat + \vec{c} + \vec{x} - \vec{y})
\; .
\label{eq:cauchyschwarz}
\end{equation}
Note that, by the right-hand side of~\eqref{eq:cauchyschwarz}, we can view $\rho(\lat + \vec{x}) \rho(\lat + \vec{y})$ as the inner product of two vectors, 
\begin{equation}
\label{eq:rho_h}
\rho(\lat + \vec{x}) \rho(\lat + \vec{y}) = \inner{\vec{h}(\vec{x} + \vec{y}), \vec{h}(\vec{x} - \vec{y})}
\; ,
\end{equation}
where
\[
h(\vec{z}) := \big( \rho_{\sqrt{2}}(2 \lat + \vec{z} + \vec{c}_1),  \rho_{\sqrt{2}}(2 \lat +\vec{z} + \vec{c}_2), \ldots, \rho_{\sqrt{2}}(2\lat + \vec{z} + \vec{c}_{2^n}) \big) \in \R^{2^n}
\; ,
\]
for some ordering of the cosets $\vec{c}_i \in \lat/(2\lat)$.
Then, by Cauchy-Schwarz, we have
\[
\rho(\lat + \vec{x})^2 \rho(\lat + \vec{y})^2 \leq  \length{\vec{h}(\vec{x} + \vec{y})}^2\length{\vec{h}(\vec{x} - \vec{y})}^2 = \rho(\lat)^2 \rho(\lat + \vec{x} + \vec{y}) \rho(\lat + \vec{x} - \vec{y}) \; ,
\]
where the last equality follows from plugging in $\vec y = \vec 0$ to Eq.~\eqref{eq:rho_h} which tells us that $\|\vec{h}(\vec{z})\|^2 = \rho(\lat) \rho(\lat + \vec{z})$.
\end{proof}

We remark that using the same proof with other transformations $T$ might
lead to other such inequalities. We leave this for future
work and proceed to list a few immediate corollaries of 
Theorem~\ref{thm:awesometheorem}.

\begin{corollary}
	\label{cor:main_theorem}
For any lattice $\lat \subset \R^n$ and any two vectors $\vec{x}, \vec{y} \in \R^n$, we have
\begin{subequations}
\label{eqs:corollaries}
\begin{align}
\label{eq:periodicform}
f_\lat(\vec{x})^2 f_\lat(\vec{y})^2 &\leq f_\lat(\vec{x} + \vec{y})f_\lat(\vec{x} - \vec{y})\\
\label{eq:f2x}
f_\lat(\vec{x})^4 &\leq f_\lat(2\vec{x})\\
\label{eq:additiveform}
f_\lat(\vec{x})f_\lat(\vec{y}) &\leq (f_\lat(\vec{x} + \vec{y}) + f_\lat(\vec{x} - \vec{y}))/2\\
\label{eq:strongcoscorrelation}
\expect_{\vec{w} \sim D_{\lat}}[\cos(2\pi \inner{\vec{w}, \vec{x}})]^2\expect_{\vec{w} \sim D_{\lat}}[\cos(2\pi \inner{\vec{w}, \vec{y}})]^2 &\leq \expect_{\vec{w} \sim D_{\lat}}[\cos(2\pi \inner{\vec{w}, \vec{x}})\cos(2\pi \inner{\vec{w}, \vec{y}})]^2\\
&\qquad - \expect_{\vec{w} \sim D_{\lat}}[\sin(2\pi \inner{\vec{w}, \vec{x}})\sin(2\pi \inner{\vec{w}, \vec{y}})]^2 \nonumber
\\
\label{eq:coscorrelation}
\expect_{\vec{w} \sim D_{\lat}}[\cos(2\pi \inner{\vec{w}, \vec{x}})]\expect_{\vec{w} \sim D_{\lat}}[\cos(2\pi \inner{\vec{w}, \vec{y}})] &\leq \expect_{\vec{w} \sim D_{\lat}}[\cos(2\pi \inner{\vec{w}, \vec{x}})\cos(2\pi \inner{\vec{w}, \vec{y}})] \; .
\end{align}
\end{subequations}
\end{corollary}
\begin{proof}
Eq.~\eqref{eq:periodicform} follows from the definition of $f_\lat$. Eq.~\eqref{eq:f2x} follows from plugging in $\vec{y} = \vec{x}$ to Eq.~\eqref{eq:periodicform}. Eq.~\eqref{eq:additiveform} follows from the fact that $\sqrt{ab} \leq (a+b)/2$ for all $a, b \ge 0$. 
For Eq.~\eqref{eq:strongcoscorrelation}, use the Poisson summation formula to write $f_{\lat^*}(\vec{x})$ in its dual form as
\begin{equation*}
f_{\lat^*}(\vec{x}) = \expect_{\vec{w} \sim D_{\lat}}[\cos(2\pi \inner{\vec{w}, \vec{x}})]
\;,
\end{equation*}
where $\lat^*$ is the dual lattice. 
We can then apply the identity $\cos(a + b) = \cos(a)\cos(b) - \sin(a) \sin(b)$ to derive Eq.~\eqref{eq:strongcoscorrelation} from Eq.~\eqref{eq:periodicform}. Finally, Eq.~\eqref{eq:coscorrelation} follows from applying the same analysis to \eqref{eq:additiveform}. 
\end{proof}

\section{Moments of the discrete Gaussian distribution}
\label{sec:moments}

We will need the Hessian product identity
\begin{equation}
\label{eq:Hess_prod}
H \big(f(\vec{x})  g(\vec{x})\big) = f(\vec{x}) H g(\vec{x}) + g(\vec{x}) H f(\vec{x}) + \grad f(\vec{x}) (\grad g(\vec{x}))^T + \grad g(\vec{x}) (\grad f(\vec{x}))^T
\; .
\end{equation}

We next show an inequality concerning the Hessian of $f_{\lat}$. 
In particular, this inequality constrains the shape of the local maxima of $f_\lat$. 
(As observed in \cite{cvpp}, $f_\lat$ can in fact have local maxima at non-lattice points.)

\begin{proposition}
\label{prop:Hess}
For any lattice $\lat \subset \R^n$ and any vector $\vec{x} \in \R^n$, we have the positive semidefinite inequality 
\[
 \frac{Hf_\lat(\vec{x})}{f_\lat(\vec{x})} \succeq  Hf_\lat (\vec0) + \frac{\grad f_\lat(\vec{x}) (\grad f_\lat(\vec{x}))^T}{f_\lat(\vec{x})^2} \; .
\]
\end{proposition}
\begin{proof}
By Eq.~\eqref{eq:periodicform}, we have
\[
f_\lat(\vec{x} + \vec{y})f_\lat(\vec{x} - \vec{y})-f_\lat(\vec{x})^2f_\lat(\vec{y})^2 \geq 0
\; .
\]
Note that we have equality when $\vec{y} = \vec0$. It follows that, for any $\vec{x}$, the left-hand side has a local minimum at $\vec{y} = \vec0$, and therefore the Hessian with respect to $\vec{y}$ at $\vec0$ must be positive semidefinite. The result follows by 
using Eq.~\eqref{eq:Hess_prod} to take the Hessian and rearranging.
\end{proof}

As a corollary, we obtain that the covariance matrix of $D_{\lat + \vec{x}}$ is minimized 
at $\vec{x}=\vec{0}$. (Notice that the expectation of the centered Gaussian $D_{\lat}$ is zero because the lattice is symmetric.) The corollary follows immediately from Proposition~\ref{prop:Hess} and the following two identities:
\begin{align}
\frac{\grad f_\lat(\vec{x})}{f_\lat(\vec{x})} &= 
\frac{\grad \rho(\lat+\vec{x})}{\rho(\lat+\vec{x})} = 
-2\pi \expect_{\vec{w} \sim D_{\lat + \vec{x}}}[\vec{w}] \; , \text{ and}  \label{eq:gradf}
\\
\frac{H f_\lat(\vec{x})}{f_\lat(\vec{x})} &= 
\frac{H \rho(\lat+\vec{x})}{\rho(\lat+\vec{x})} = 
4\pi^2 \expect_{\vec{w} \sim D_{\lat + \vec{x}}}[\vec{w}\vec{w}^T] - 2\pi I_n \label{eq:Hessf}
\; .
\end{align}

\begin{corollary}
\label{cor:secondmoment}
For any lattice $\lat \subset \R^n$ and vector $\vec{x} \in \R^n$, we have the positive semidefinite inequality 
\[
\expect_{\vec{w} \sim D_{\lat + \vec{x}}}[\vec{w}\vec{w}^T] - \expect_{\vec{w} \sim D_{\lat + \vec{x}}}[\vec{w}]\expect_{\vec{w} \sim D_{\lat + \vec{x}}}[\vec{w}^T]
\succeq 
\expect_{\vec{w} \sim D_{\lat}}[\vec{w}\vec{w}^T] 
\; .
\]
In particular,
\[
\expect_{\vec{w} \sim D_{\lat + \vec{x}}}[\length{\vec{w}}^2]  - \Big\|\expect_{\vec{w} \sim D_{\lat + \vec{x}}}[\vec{w}]\Big\|^2 \geq \expect_{\vec{w} \sim D_\lat}[\length{\vec{w}}^2] 
\; .
\]
\end{corollary}

The following proposition (with $\vec u = \vec v$) implies that the one-dimensional projections 
of the discrete Gaussian distribution are ``leptokurtic,'' i.e., have 
kurtosis at least $3$, the kurtosis of a normal variable. 
We remark that the case $n=1$ follows from a known inequality 
related to the Riemann zeta function~\cite{Chung76,Newman76} (see also~\cite[Section 2.2]{BianePY01}). 

\begin{proposition}
\label{prop:fourthmoment}
For any lattice $\lat \subset \R^n$ and vectors $\vec{u},\vec{v} \in \R^n$,
\[
\expect_{\vec{y} \sim D_{\lat}}[\inner{\vec{y},\vec{u}}^2\inner{\vec{y},\vec{v}}^2]  \geq \expect_{\vec{y} \sim D_{\lat}}[\inner{\vec{y}, \vec{u}}^2 ]\expect_{\vec{y} \sim D_{\lat}}[\inner{\vec{y}, \vec{v}}^2 ] + 2\expect_{\vec{y} \sim D_{\lat}}[\inner{\vec{y},\vec{u}}\inner{\vec{y},\vec{v}}]^2 
\; .
\]
\end{proposition}
\begin{proof}
	From Corollary~\ref{cor:secondmoment}, we have
	\[
	\expect_{\vec{w} \sim D_{\lat + \vec{x}}}[\inner{\vec{w}, \vec{u}}^2] - \expect_{\vec{w} \sim D_{\lat + \vec{x}}}[\inner{\vec{w}, \vec{u}}]^2 - \expect_{\vec{w} \sim D_{\lat}}[\inner{\vec{w}, \vec{u}}^2] \geq 0
		\; .
	\]
	Note that the left-hand side equals zero when $\vec x = \vec0$ since $\lat = - \lat$. 
	The same is true if we multiply through by $\rho(\lat + \vec{x})^2$, which leads to the inequality
	\[
	\sum_{\vec{y}, \vec{y}' \in \lat} \rho(\vec{y} + \vec{x})\rho(\vec{y}' + \vec{x}) \cdot \Big( \inner{\vec{y} - \vec{y}', \vec{u}}^2/2 - \expect_{\vec{w} \sim D_{\lat}}[\inner{\vec{w}, \vec{u}}^2] \Big)  \geq 0
	\; .
	\]
	(To see that, use $\inner{\vec{y} - \vec{y}', \vec{u}} = \inner{\vec{y} +\vec{x}, \vec{u}} - \inner{\vec{y}'+\vec{x}, \vec{u}}$, and expand the square.) Therefore, as in the proof of Proposition~\ref{prop:Hess}, the Hessian of the left-hand side with respect to $\vec{x}$ at $\vec{x} = \vec0$ must be positive semidefinite.
	Using Eqs.~\eqref{eq:Hess_prod},~\eqref{eq:gradf}, and~\eqref{eq:Hessf}, we see that
	\begin{align*}
		H(\rho(\vec{y} + \vec{x}) \rho(\vec{y}'  +\vec{x})) |_{\vec{x} = \vec0} = 4\pi^2 \rho(\vec{y}) \rho(\vec{y}') \big( (\vec{y} + \vec{y}')(\vec{y} + \vec{y}')^T - I_n/\pi \big)
		\; .
	\end{align*}
	Therefore,
	\begin{align*}
		0 
		&\preceq \expect_{\vec{y}, \vec{y}' \sim D_{\lat}} \Big[  \big( (\vec{y} + \vec{y}')(\vec{y} + \vec{y}')^T  - I_n/\pi \big) \cdot \Big( \inner{\vec{y} - \vec{y}', \vec{u}}^2/2 - \expect_{\vec{w} \sim D_{\lat}}[\inner{\vec{w}, \vec{u}}^2] \Big) \Big] \\
		&= \expect_{\vec{y}, \vec{y}' \sim D_{\lat}} \Big[  (\vec{y} + \vec{y}')(\vec{y} + \vec{y}')^T  \cdot \Big( \inner{\vec{y} - \vec{y}', \vec{u}}^2/2 - \expect_{\vec{w} \sim D_{\lat}}[\inner{\vec{w}, \vec{u}}^2] \Big) \Big] \\
		&= \expect_{\vec{y}, \vec{y}' \sim D_{\lat}} \Big[ \vec{y}\vec{y}^T\inner{\vec{y}, \vec{u}}^2 + \vec{y} \vec{y}^T \inner{\vec{y}', \vec{u}}^2 - (\vec{y} \vec{y}^{\prime T} + \vec{y}' \vec{y}^T)\inner{\vec{y}, \vec{u}}\inner{\vec{y}', \vec{u}} - 2 \vec{y} \vec{y}^T \expect_{\vec{w} \sim D_{\lat}}[\inner{\vec{w}, \vec{u}}^2] \Big]\\
		&= \expect_{\vec{y} \sim D_{\lat}}[\vec{y}\vec{y}^T\inner{\vec{y}, \vec{u}}^2] - \expect_{\vec{y} \sim D_{\lat}}[\vec{y} \vec{y}^T ]\expect_{\vec{y} \sim D_{\lat}}[\inner{\vec{y}, \vec{u}}^2] - 2\expect_{\vec{y} \sim D_{\lat}}[\vec{y} \inner{\vec{y}, \vec{u}}] \expect_{\vec{y} \sim D_{\lat}}[\vec{y}^T\inner{\vec{y}, \vec{u}}]
		\; ,
	\end{align*}
	as needed.
\end{proof}

\section{Monotonicity of the periodic Gaussian function}
\label{sec:monotone}

The next proposition shows that $f_{\lat, s}(\vec{x})$ is non-decreasing as a function of $s$.
This (and the more general statement in Proposition~\ref{prop:monocovariance}) answers 
a question of Price~\cite{PriceMO1}, who proved it for the one-dimensional case $n=1$ 
(illustrated in Figure~\ref{fig:monoton}). 

One might wonder if such a monotonicity property is specific to flat tori or whether
it is a special case of a more general phenomenon. 
Namely, Peres~\cite{PeresPrivate} asked whether for any vertex transitive graph $G$ it holds that for any 
two vertices $u,v$, the ratio $\Pr[X_t = v]/\Pr[X_t = u]$ is non-decreasing 
as a function of $t$, 
where $X_t$ is a continuous-time random walk on $G$ starting at $u$ after time $t$.
Recently, using our result, Price showed how to prove this for Abelian Cayley graphs~\cite{Price16}.
Interestingly, a further extension to arbitrary Cayley graphs turns out to be false~\cite{RegevShinkar}.

\begin{figure}[ht]
\begin{center}
\includegraphics[width= 0.7\textwidth]{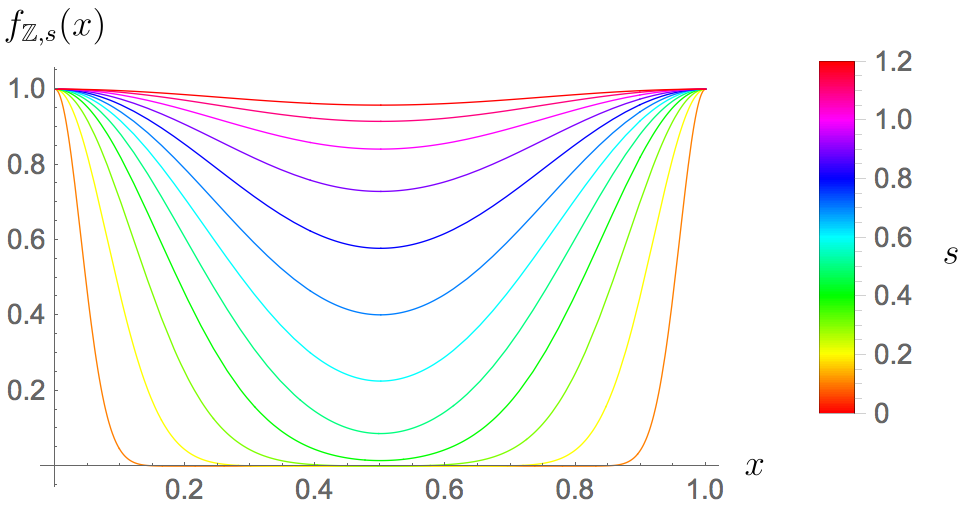} 
\caption{\label{fig:monoton} 
$f_{\Z,s}(x)$ for various values of $s$ and $x \in [0,1]$.}
\end{center}
\end{figure}

\begin{proposition}
\label{prop:monotones}
For any lattice $\lat \subset \R^n$ and vector $\vec{x} \in \R^n$, 
\[
\frac{\frac{d}{ds} f_{\lat, s}(\vec{x})}{f_{\lat, s}(\vec{x})}\geq  
\frac{s}{2\pi} \cdot \frac{\length{\grad f_{\lat, s}(\vec{x})}^2}{f_{\lat, s}(\vec{x})^2} 
\; .
\]
\end{proposition}

\begin{proof}
A straightforward computation shows that 
\begin{align*}
\frac{d}{ds} f_{\lat, s}(\vec{x}) &= \frac{2\pi f_{\lat, s}(\vec{x})}{s^3} \cdot  \expect_{\vec{w} \sim D_{\lat + \vec{x}, s}}[\length{\vec{w}}^2] - \frac{2\pi f_{\lat, s}(\vec{x})}{s^3} \cdot \expect_{\vec{w} \sim D_{\lat, s}}[\length{\vec{w}}^2] \\
&\geq \frac{2\pi f_{\lat, s}(\vec{x})}{s^3} \length[\Big]{\expect_{\vec{w} \sim D_{\lat + \vec{x}, s}}[\vec{w}]}^2
\; ,
\end{align*}
where we have applied Corollary~\ref{cor:secondmoment}. The result then follows from the fact that (see~Eq.\eqref{eq:gradf})
\[ 
\frac{\grad f_{\lat, s}(\vec{x})}{f_{\lat, s}(\vec{x})} = 
-\frac{2\pi}{s^2} \cdot\expect_{\vec{w} \sim D_{\lat + \vec{x},s}}[\vec{w}] 
\; . \qedhere
\]
\end{proof}

We now extend this monotonicity result by replacing the scalar variance parameter $s^2$ by a positive-definite matrix $\Sigma$. In particular, we define 
	\[
	f_{\lat, \Sigma}(\vec{x}) = \frac{\sum_{\vec{y} \in \lat} \exp(-\pi (\vec{y} + \vec{x})^T \Sigma^{-1} (\vec{y} + \vec{x}))}{\sum_{\vec{y} \in \lat} \exp(-\pi \vec{y}^T \Sigma^{-1} \vec{y})}
	\; .
	\]
	Equivalently, 
	\[
	f_{\lat, \Sigma}(\vec{x}) = f_{\Sigma^{-1/2} \lat}(\Sigma^{-1/2}\vec{x})
	\; ,
	\]
	where $\Sigma^{1/2}$ is the unique positive-definite square root of $\Sigma$.

\begin{proposition}
	\label{prop:monocovariance}
	For any lattice $\lat \subset \R^n$, $\vec{x} \in \R^n$, and positive-definite matrices $\Sigma, \Sigma' \in \R^{n \times n}$ satisfying the positive semidefinite inequality 
	$\Sigma' \preceq \Sigma$,
	\[
	f_{ \lat, \Sigma'}( \vec{x}) \leq f_{ \lat, \Sigma}(\vec{x})
	\; .
	\]
\end{proposition}
\begin{proof}
	 We may replace $\lat$ by $\Sigma^{\prime -1/2}\lat$, $\vec{x}$ by $\Sigma^{\prime -1/2} \vec{x}$, 
	and $\Sigma$ by $\Sigma^{\prime -1/2} \Sigma \Sigma^{\prime -1/2} $ so that we can assume without loss of generality that $\Sigma' = I_n$. Moreover, by a change of basis, we may take $\Sigma$ to be diagonal. (Here, we have used the fact that the Gaussian is invariant under orthogonal transformations.)
	
	So, it suffices to show that $f_{\lat}(\vec{x}) \leq f_{\lat, \Sigma}(\vec{x})$ when $\Sigma \in \R^{n\times n}$ is a diagonal matrix with $\Sigma \succeq I_n$. Let $s_1^2, \ldots, s_n^2 \ge 1$ be the entries along the diagonal of $\Sigma$. The proof now proceeds nearly identically to the proof of Proposition~\ref{prop:monotones}. Differentiating with respect to $s_i$, we have
	\[
	\frac{{\rm d} }{{\rm d} s_i} f_{\lat, \Sigma}(\vec{x}) = 
	\frac{2\pi f_{ \lat, \Sigma }( \vec{x})}{s_i^3} 
	\Big(
	\expect_{\vec{w} \sim D_{\Sigma^{-1/2}(\lat + \vec{x})}}[w_i^2] - 
	\expect_{\vec{w} \sim D_{\Sigma^{-1/2}\lat }}[w_i^2]
	\Big)
	\; ,
	\]
	where $w_i$ is the $i$th coordinate of $\vec{w}$. The result follows by noting that Corollary~\ref{cor:secondmoment} implies that this derivative is positive for all $s_i > 0$, so that $f_{\lat, \Sigma}(\vec{x})$ is an increasing function of $s_i$.
\end{proof}

We next give another monotonicity result, now with respect to taking sublattices. 

\begin{proposition}
\label{prop:monotonesublattice}
For any lattice $\lat \subset \R^n$, sublattice $\M \subseteq \lat$, and vector $\vec{x} \in \R^n$,
\[
f_{\M}(\vec{x}) \leq f_{\lat}(\vec{x}) \; .
\]
\end{proposition}
\begin{proof}
\begin{align*}
\rho(\M + \vec{x}) \rho(\lat) &= \sum_{\vec{c} \in \lat/\M}\rho(\M + \vec{x}) \rho(\M + \vec{c}) \\
&\leq \sum_{\vec{c} \in \lat/\M} \rho(\M) (\rho(\M + \vec{x} + \vec{c}) + \rho(\M + \vec{x} - \vec{c}))/2 &\text{(Eq.~\eqref{eq:additiveform})}\\
&= \sum_{\vec{c} \in \lat/\M} \rho(\M) \rho(\M + \vec{x} + \vec{c}) \\
&= \rho(\M)\rho(\lat + \vec{x})
\; .
\end{align*}
The result follows.
\end{proof}

\section{Positive correlation of the Gaussian measure on lattices}
\label{sec:correlation}

The following shows that 
sublattices are positively correlated under the normalized Gaussian measure on a lattice. (Price asked whether this holds in the special case when $\NN := \lat \cap V$ for some subspace $V \subseteq \R^n$~\cite{PriceMO2}.)

\begin{theorem}
\label{thm:sublatticecorrelation}
For any lattice $\lat \subset \R^n$ and sublattices $\M,\NN \subseteq \lat$,
\[
\frac{\rho(\M)}{\rho(\lat)} \cdot \frac{\rho(\NN)}{\rho(\lat)}
 \leq \frac{\rho(\M \cap \NN)}{\rho(\lat)}
\; .
\]
\end{theorem}
\begin{proof}
Note that the natural mapping from $\M/(\M \cap \NN)$ to $\lat/\NN$ given by $\vec{c} \mapsto \NN + \vec{c}$ is injective. 
So,
\begin{align*}
\frac{\rho(\lat)}{\rho(\NN)} &= \sum_{\vec{c} \in \lat/\NN} \frac{\rho( \NN + \vec{c})}{\rho(\NN)}\\
&\geq  \sum_{\vec{c} \in \M/(\M \cap \NN)} \frac{\rho(\NN + \vec{c})}{\rho(\NN)} \\
&\geq \sum_{\vec{c} \in \M/(\M \cap \NN)} \frac{\rho((\M \cap \NN) + \vec{c})}{\rho(\M \cap \NN)} &\text{(Prop.~\ref{prop:monotonesublattice})}\\
&= \frac{\rho(\M)}{\rho(\M \cap \NN)}
\; .
\end{align*}
The result follows by rearranging.
\end{proof}

\subsubsection*{Acknowledgements } We thank Tom Price for helpful discussion, and the anonymous referees for their comments.

\bibliographystyle{alpha}

\end{document}